\documentclass[reqno, 11pt]{amsart}

\newif\ifextract
\extractfalse

\newif\ifdefs
\defsfalse

\newif\ifdiscs 
\discstrue

\newenvironment{extradetails}{
	\ifdefs \noindent \newline \noindent \bf
	************************** begin extra details\rm \\ \rm\fi 
}{
	\ifdefs \rm \noindent \newline
	\noindent \bf ************************** end extra details \rm\\ \fi
}

\usepackage{univmgonalpreamble}

\ifextract

\ifdiscs
\ifdefs
\usepackage[active,generate=univmgonal_details21,extract-env={extradetails,discussion,extractsection}]{extract}
\else
\usepackage[active,generate=univmgonal_details21,extract-env={discussion,extractsection}]{extract}
\fi 
\else 
\ifdefs
\usepackage[active,generate=univmgonal_details21,extract-env={extradetails,extractsection}]{extract}
\fi
\fi

\fi
\makeatother
 
\author{Soumyarup Banerjee}
\address{Department of Mathematics, Indian Institute of technology Kharagpur, Kharagpur, West Bengal - 721302, India}
\email{soumya.tatan@gmail.com}
\author{Ben Kane}
\address{Mathematics Department, The University of Hong Kong, Pokfulam, Hong Kong}
\email{bkane@hku.hk}
\author{Kwan To Ng}
\address{Mathematics Department, The University of Hong Kong, Pokfulam, Hong Kong}
\email{u3011714@connect.hku.hk}

\title{Universal sums of generalized polygonal numbers of almost prime length}
\date{\today}
\thanks{The research of the first author was supported by INSPIRE Faculty Research Grant by DST India, Grant no: DST/INSPIRE/04/2021/002753. The research of the second author was supported by grants from the Research Grants
Council of the Hong Kong SAR, China (project numbers HKU 17314122, HKU 17305923).}
\begin{document}
\begin{abstract}
In this paper, we consider universal sums of generalized polygonal numbers. Fixing $m\in\N_{\geq 3}$, we show two finiteness theorems for universal sums of generalized polygonal numbers whose inputs have a restricted number $L$ of prime divisors (counting multiplicity) away from an finite set of exceptional primes. In the first theorem, we fix $m$ and uniformly bound the finite check independent of $L\geq 900$, and in the second theorem, we give an optimal bound for the finiteness check if $L$ is larger than a constant times $\log(m)$.
\end{abstract}
\maketitle

\begin{extractsection}
\section{Introduction}
\end{extractsection}
For $m\geq 3$, let 
\[
p_m(n):=\frac{(m-2)n^2}{2} + \frac{(4-m)n}{2}
\]
be the $n$-th {\it $m$-gonal number}. The number $p_m(n)$ counts the number of dots in a regular $m$-gon of length $n$ for $n\in\N$. We allow $n\in\Z$ and call these \begin{it}generalized $m$-gonal numbers\end{it}.

For $n\in\N$, $\ell\in\N$, and $\bm{a}\in\N^{\ell}$, consider the equation
\begin{equation}\label{eqn:summgonal}
\sum_{j=1}^{\ell}a_j p_m(n_j)=n.
\end{equation}
If such a solution exists, then we say that $n$ is represented by the sum of generalized $m$-gonal numbers (with positive $a_j$ throughout). 
We would like to determine if a solution to the equation \eqref{eqn:summgonal} exists under the additional restriction that the $n_j$ are almost primes of some order $L$ up to arbitrary divisibility by primes in some fixed set $S$, i.e, the number of primes dividing $n_j$ (counting multiplicities) not contained in $S$ is bounded by $L$. We call $n_j$ a $P_{L,S}$-number (and simply write $n_j\in P_{L,S}$ for simplicity) if its prime factorization satisfies this property or $n_j=0$. 

We call a sum of generalized $m$-gonal numbers \begin{it}universal\end{it} if \eqref{eqn:summgonal} is solvable for all $n\in\N$ with $n_j\in\Z$ and \begin{it}$P_{L,S}$-universal\end{it} if \eqref{eqn:summgonal} is solvable for all $n\in\N$ with $n_j\in P_{L,S}$. The second author and Liu \cite{KaneLiu} showed that there exists $\gamma_m$ such that a generalized $m$-gonal sum is universal if and only if \eqref{eqn:summgonal} is solvable for all $n\geq \gamma_m$ and proved that $c_1 (m-2) \leq \gamma_m\leq C_{\varepsilon}(m-2)^{7+\varepsilon}$ for some absolute constsant $c_1$ and an absolute (effective) constant $C_{\varepsilon}$ only depending on $\varepsilon$; this was later improved by Kim and Park \cite{KimPark} to show that there is an absolute constant $C$ such that $c_1 (m-2)\leq \gamma_m\leq C(m-2)$. We call such bounds \begin{it}finiteness theorems\end{it} because they reduce the check for universality to a finite check. The main theorem of this paper extends the finiteness theorem for universality to one for $P_{L,S}$-universality.

\begin{theorem}\label{thm:main}
    Let $S:=\{2,3\}.$ Suppose that $m\geq 11$ is odd and $m\not\equiv4 \pmod3.$ 
    \begin{enumerate}[leftmargin=*,label=\rm(\arabic*)]
    \item If $L\geq 900,$ then we have $N_{L,S}\leq C(m-2)^{7980}$ for some absolute constant $C,$ where $N_{L,S}$ is the infimum of integers such that if $Q$ represents every $n \leq N_{L,S}$ with $x \in P_{L,S}$, then $Q$ is $P_{L,S}$-universal.
    In other words,
for $L\geq 900$, a sum of generalized $m$-gonal numbers is $P_{L,S}-$universal if and only if it represents every $n \leq C(m-2)^{7980}$ with $P_{L,S}-$numbers.
    \item Suppose that $L\geq \max\{900,1+7980\log_5(m-2)\}$. Then $N_{L,S}\leq C(m-2)$ for some absolute constant $C$.
    In other words, for these $L$, a sum of generalized $m$-gonal numbers is $P_{L,S}-$universal if and only if it represents every $n \leq C(m-2)$. Furthermore, for these $L$, a sum of generalized $m$-gonal numbers is $P_{L,S}$-universal if and only if it is universal.
\end{enumerate}
\end{theorem}

\begin{remarks}
    Note that the first part of the theorem gives a finiteness theorem with a uniform bound in $L$
    while the second part gives an optimal bound for $N_{L,S}$, at the
expense that $L$ must grow logarithmically with $m$.
\end{remarks}

The paper is organized as follows. In Section \ref{sec:prelim}, we introduce the basic argument used in the paper and recall some useful results. In Section \ref{sec:Eisenstein}, we investigate the Eisenstein series component of the relevant theta functions. We apply sieving techniques in Section \ref{sec:sieve}, and prove the main theorem in Section \ref{sec:mainproof}
\begin{extractsection}
\section{Preliminaries}\label{sec:prelim}
\end{extractsection}
\subsection{Basic Argument}
In this paper, we are interested in solutions to \eqref{eqn:summgonal} with $n_j$ chosen to have a small number of prime divisors (if $n_j\neq 0$), hopefully independent of $n$. In other words, writing $n_j=\pm \prod_{p\mid n_j}p^{b_p}$, we would like to have 
\[
\Omega(n_j):=\sum_{p\mid n_j}b_p
\]
as small as possible. Congruence issues may sometimes require us to allow some small primes to divide $n_j$ to a high power (for example, by a straightforward calculation modulo $8$ one can  show by induction that all solutions to $n_1^2+n_2^2+n_3^2+n_4^2=2^n$ satisfy $\ord_2(n_j)\geq \lfloor \frac{n-1}{2}\rfloor$). 
\begin{extradetails}\rm
\begin{claim}\label{claim:sum4squares2pow}
\end{claim}
\end{extradetails}
We hence can only in general expect that 
\begin{equation}\label{eqn:omegadef}
\omega(n_j):=\#\{p\mid n_j\}
\end{equation}
is bounded independent of $n$ and that there is some finite set of primes (independent of $n$) $\mathcal{S}$ for which 
\begin{equation}\label{eqn:OmeganoS}
\sum_{\substack{p\mid n_j\\ p\notin\mathcal{S}}}b_p
\end{equation}
is also bounded independent of $n$.

 In order to obtain solutions to \eqref{eqn:summgonal} for which both \eqref{eqn:omegadef} and \eqref{eqn:OmeganoS} are bounded independent of $n$, for some $\delta>0$ we set 
\[
M_{\mathcal{S}}(X):=\prod_{\substack{p\leq X\\ p\notin \mathcal{S}}} p
\]
and consider solutions to \eqref{eqn:summgonal} with 
\begin{equation}\label{eqn:gcdcondition}
\gcd\left(n_j, M_{\mathcal{S}}(X)\right)=1.
\end{equation}
Then $p\mid n_j$ implies that $p\in \mathcal{S}$ or $p>X$, so only ``large'' primes outside of the set $\mathcal{S}$ may divide $n_j$.

Moreover, since solutions to \eqref{eqn:summgonal} satisfy $|n_j|\ll_m \sqrt{n}$,
we have 
\[
\sqrt{n}\gg_m |n_j|=\prod_{\substack{p\mid n_j\\ p\in\mathcal{S}}} p^{b_p} \prod_{\substack{p\mid n_j\\ p\notin \mathcal{S}}} p^{b_p}> \prod_{\substack{p\mid n_j\\ p\notin \mathcal{S}}} X^{b_p}=X^{\sum_{\substack{p\mid n_j\\ p\notin \mathcal{S}}} b_p}.
\]
In particular, if $X=n^{\delta}$ for some $\delta>0$, then 
\[
\sqrt{n}\gg_m n^{\delta\sum_{\substack{p\mid n_j\\ p\notin \mathcal{S}}} b_p}.
\]
We conclude that
\begin{equation}\label{eqn:almostprimebound}
\sum_{\substack{p\mid n_j\\ p\notin\mathcal{S}}}b_p\ll_m \frac{1}{2\delta}.
\end{equation}
This gives a bound on \eqref{eqn:OmeganoS} independent of $n$. Moreover, since $b_p\geq 1$, \eqref{eqn:almostprimebound} implies that we have
\[
\#\{p\mid n_j\}=\sum_{p\mid n_j} 1 = \sum_{\substack{p\mid n_j\\ p\in\mathcal{S}}} 1 + \sum_{\substack{p\mid n_j\\ p\notin\mathcal{S}}} 1\leq \#\mathcal{S} + \sum_{\substack{p\mid n_j\\ p\notin\mathcal{S}}} b_p\ll_{m} \#\mathcal{S} + \frac{1}{2\delta},
\]
giving a bound on \eqref{eqn:omegadef} independent of $n$ (as $\mathcal{S}$ is assumed to be independent of $n$). So it suffices to investigate solutions to \eqref{eqn:summgonal} with $n_j$ satisfying \eqref{eqn:gcdcondition} (for those $j$ with $n_j\neq 0$).

\subsection{Shifted lattices and sums of generalized polygonal numbers}
In order to investigate solutions to problems like \eqref{eqn:summgonal} with $n_j$ satisfying gcd conditions such as in \eqref{eqn:gcdcondition}, it is common to instead consider solutions to \eqref{eqn:summgonal} with $d_j\mid \gcd(n_j,M_{\mathcal{S}}(X))$ and then apply sieving techniques, following techniques developed by Br\"udern and Fouvry \cite{BruedernFouvry}. Completing the square, we write 
\[
	p_m(n) = \frac1{8(m-2)}(2(m-2)n+4-m)^2 - \frac{(m-4)^2}{8(m-2)}.
\]
Thus \eqref{eqn:summgonal} becomes
\begin{align*}
\sum_{j=1}^{\ell}a_j\left( 2(m\!-\!2)n_j +4-m\right)^2= \sum_{j=1}^\ell a_j\!\left(8(m\!-\!2)p_m(n_j)+(m\!-\!4)^2\right)
 =8(m\!-\!2)n+(m\!-\!4)^2\sum_{j=1}^{\ell} a_j.
\end{align*}
So \eqref{eqn:summgonal} is equivalent to
\begin{equation}\label{eqn:cong}
\sum_{j=1}^{\ell} a_j N_j^2=8(m-2)n+(m-4)^2\sum_{j=1}^{\ell}a_j,
\end{equation}
where $N_j:=2(m-2)n_j+4-m$. In order to restrict $n_j$ to be almost primes, we assume that many primes $p$ do not divide $\prod_{j=1}^{\ell} n_j$. For $p\nmid 2(m-2)$, this restriction becomes 
\begin{equation*}
N_j\not\equiv 4-m\pmod{p}.
\end{equation*}

We therefore begin by considering solutions to \eqref{eqn:summgonal} with $d_j\mid n_j$ for different choices of $d_j\mid M_{\mathcal{S}}(X)$ (note that since $M_{\mathcal{S}}(X)$ is squarefree, so is $d_j$). To do so, we next rewrite the problem in terms of counting vectors on quadratic lattices.

Let $L$ be the lattice with Gram matrix $\left<a_1,a_2,\cdots,a_{\ell}\right>$ with respect to the basis $\bm{e_1},\dots,\bm{e_{\ell}}$, i.e., $Q(\bm{e_j})=a_j$ and the $\bm{e_j}$ are orthogonal to each other with respect to the associated bilinear form (see \cite{OMeara} for an introduction to the theory of lattices). For $\bm{d}\in\N^{\ell}$, we set
\begin{equation}\label{eqn:vdjdef}
\bm{v_{d,j}}:=(m-2)d_j\bm{e_j}
\end{equation}
so that the lattice $L_{\bm{d}}$ spanned by $\bm{v_{d,1}}, \dots, \bm{v_{d,\ell}}$ has Gram matrix \[
\left<a_1(m-2)^2d_1^2,\dots,a_{\ell}(m-2)^2d_{\ell}^2\right>.
\]In other words, using the orthogonality of the $\bm{e_j}$, for $\bm{x}\in\Q^{\ell}$ we have 
\begin{equation}\label{eqn:QLd+vdj}
Q\left(\sum_{j=1}^{\ell} x_j \bm{v_{d,j}}\right) =\sum_{j=1}^{\ell} Q\left(\bm{v_{d,j}}\right) x_j^2 = \sum_{j=1}^\ell Q(\bm{e_j})(m-2)^2d_j^2x_j^2  =\sum_{j=1}^{\ell} a_j (m-2)^2d_j^2x_j^2. 
\end{equation}
Set 
\begin{equation}\label{eqn:nuddef}
\bm{\nu}_{\bm{d}}:=\frac{(4-m)}{2(m-2)}\sum_{j=1}^{\ell}\frac{1}{d_j}\bm{v_{d,j}}=\frac{4-m}{2}\sum_{j=1}^{\ell}\bm{e_j}\in \frac{1}{2}L.
\end{equation}
Points on the shifted lattice $L_{\bm{d}}+\bm{\nu}_{\bm{d}}$ are related to solutions to \eqref{eqn:summgonal} with $d_j\mid n_j$, as evidenced in the following lemma.
\begin{lemma}\label{lem:Lvd+vdj}
For $n\in\N$, we have 
\begin{multline*}
\#\left\{\bm{n}\in\Z^{\ell}: \sum_{j=1}^{\ell} a_jp_{m}\left(n_j\right)=n,\ d_j\mid n_j\right\}\\
 =\#\left\{\bm{v}\in L_{\bm{d}}+\bm{\nu}_{\bm{d}}: Q(\bm{v})=2n(m-2) +\left(\frac{m-4}{2}\right)^2\sum_{j=1}^{\ell}a_j\right\}.
\end{multline*}
\end{lemma}
\begin{proof}
We take 
\[
x_j=n_j +\frac{4-m}{2(m-2)d_j}
\]
with $n_j\in\Z$ in \eqref{eqn:QLd+vdj} to obtain that
\begin{align*}
Q\left(\sum_{j=1}^{\ell} \left(n_j+\frac{4-m}{2(m-2)d_j}\right) \bm{v_{d,j}}\right)&=\sum_{j=1}^{\ell} a_j(m-2)^2 d_j^2\left(n_j+\frac{4-m}{2(m-2)d_j}\right)^2\\
&=\frac{1}{4}\sum_{j=1}^{\ell} a_j \left(2(m-2) d_j n_j+4-m\right)^2.
\end{align*}
Hence 
\begin{equation}\label{eqn:congruence}
\sum_{j=1}^{\ell} a_j \left(2(m-2) d_j n_j+4-m\right)^2=8(m-2)n+(m-4)^2\sum_{j=1}^{\ell}a_j
\end{equation} 
if and only if
\begin{align}\nonumber
Q\left(\sum_{j=1}^{\ell}\left(n_j+\frac{4-m}{2(m-2)d_j}\right)\bm{v_{d,j}}\right)&=\frac{8(m-2)n+(m-4)^2\sum_{j=1}^{\ell}a_j}{4}\\
&=2n(m-2) +\left(\frac{m-4}{2}\right)^2\sum_{j=1}^{\ell}a_j.\label{eqn:shiftedlattice}
\end{align}
Setting
\[
\mathcal{S}_{\bm{d}}=\mathcal{S}_{m,\bm{d}}:=\left\{\sum_{j=1}^{\ell}\left(n_j+\frac{4-m}{2(m-2)d_j}\right)\bm{v_{d,j}}: \bm{n}\in\Z^{\ell}\right\},
\]
we conclude from the equivalence of \eqref{eqn:congruence} and \eqref{eqn:shiftedlattice} that  
\begin{multline}\label{eqn:calSdcong}
\#\left\{\bm{n}\in\Z^{\ell}: \sum_{j=1}^{\ell} a_j \left(2(m-2) d_j n_j+4-m\right)^2=8(m-2)n+(m-4)^2\sum_{j=1}^{\ell}a_j\right\}\\
=\#\left\{\bm{v}\in \mathcal{S}_{\bm{d}}: Q(\bm{v})=2n(m-2) +\left(\frac{m-4}{2}\right)^2\sum_{j=1}^{\ell}a_j\right\}.
\end{multline}
Expanding
\[
\sum_{j=1}^{\ell}\left(n_j+\frac{4-m}{2(m-2)}\right)\bm{v_{d,j}}=\sum_{j=1}^{\ell} n_j\bm{v_{d,j}} + \sum_{j=1}^{\ell}\frac{4-m}{2(m-2)}\bm{v_{d,j}}
\]
and recalling that (by definition \eqref{eqn:nuddef})
\[
\bm{\nu}_{\bm{d}}=\sum_{j=1}^{\ell}\frac{4-m}{2(m-2)d_j}\bm{v_{d,j}},
\]
we have 
\[
\mathcal{S}_{\bm{d}}=\left\{\sum_{j=1}^{\ell}n_j\bm{v_{d,j}} + \bm{\nu}_{\bm{d}}:\bm{n}\in\Z^{\ell}\right\}.
\]
Since $\bm{v_{d,j}}$ are a $\Z$-basis for the lattice $L_{\bm{d}}$, we have 
\[
\left\{\sum_{j=1}^{\ell} n_j\bm{v_{d,j}}:\bm{n}\in\Z^{\ell}\right\}=L_{\bm{d}},
\]
and hence 
\[
\mathcal{S}_{\bm{d}}=L_{\bm{d}}+\bm{\nu}_{\bm{d}}.
\]
Plugging this into \eqref{eqn:calSdcong} and writing $N_j=2(m-2)d_jn_j+4-m$ as in \eqref{eqn:cong} (with $N_j\equiv 4-m\pmod{2(m-2)d_j}$) we conclude that  
\begin{multline}\label{eqn:congshiftequiv}
\#\left\{\bm{N}\in\Z^{\ell}: N_j\equiv 4-m\pmod{2(m-2)d_j},\ \sum_{j=1}^{\ell}a_jN_j^2 = 8(m-2)n+(m-4)^2\sum_{j=1}^{\ell}a_j\right\}\\
=\#\left\{\bm{v}\in L_{\bm{d}}+\bm{\nu}_{\bm{d}}: Q(\bm{v})=2n(m-2) +\left(\frac{m-4}{2}\right)^2\sum_{j=1}^{\ell}a_j\right\}.
\end{multline}
As noted below \eqref{eqn:cong},  the equivalence between \eqref{eqn:summgonal} and \eqref{eqn:cong} is given by $N_j=2(m-2)n_j+4-m$. 
The condition $N_j\equiv 4-m\pmod{2(m-2)d_j}$ in \eqref{eqn:cong} (and consequently on the left-hand side of \eqref{eqn:congshiftequiv}) is hence equivalent to $n_j=d_jx_j$ with $x_j\in\Z$, or in other words $d_j\mid n_j$ in \eqref{eqn:summgonal}. 
 Hence $N_j\equiv 4-m\pmod{2(m-2)d_j}$ if and only if $d_j\mid n_j$, and we see that 
\begin{multline*}
\left\{\bm{N}\in\Z^{\ell}: N_j\equiv 4-m\pmod{2(m-2)d_j}\, \sum_{j=1}^{\ell}a_jN_j^2 = 8(m-2)n+(m-4)^2\sum_{j=1}^{\ell}a_j\right\}\\
=\#\left\{\bm{n}\in\Z^{\ell}: \sum_{j=1}^{\ell}a_jp_m(n_j)=n, d_j\mid n_j\right\}.
\end{multline*}
Plugging this into \eqref{eqn:congshiftequiv} yields 
\begin{align*}
&\hspace{-.7cm}\#\left\{\bm{n}\in\Z^{\ell}: \sum_{j=1}^{\ell}a_jp_m(n_j)=n, d_j\mid n_j\right\}\\
=&\,\#\left\{\bm{N}\in\Z^{\ell}: N_j\equiv 4-m\pmod{2(m-2)d_j}\, \sum_{j=1}^{\ell}a_jN_j^2 = 8(m-2)n+(m-4)^2\sum_{j=1}^{\ell}a_j\right\}\\
=&\,\#\left\{\bm{v}\in L_{\bm{d}}+\bm{\nu}_{\bm{d}}: Q(\bm{v})=2n(m-2) +\left(\frac{m-4}{2}\right)^2\sum_{j=1}^{\ell}a_j\right\}.\qedhere
\end{align*}
\end{proof}
Based on Lemma \ref{lem:Lvd+vdj}, in order to investigate the number of solutions to \eqref{eqn:summgonal} with $d_j\mid n_j$, we may use results from \cite{KaneLiu} and \cite{KaneYang} to investigate  
\[
\#\left\{\bm{v}\in L_{\bm{d}}+\bm{\nu}_{\bm{d}}: Q(\bm{v})=2n(m-2) +\left(\frac{m-4}{2}\right)^2\sum_{j=1}^{\ell}a_j\right\}.
\]

\subsection{Modular forms and theta functions}
Given the connection with shfited lattices $X=L+\nu$, we consider the problem of representing $n$ as a sum of polygonal numbers through the \begin{it}theta function\end{it}
\[
\Theta_{X}(\tau):=\sum_{\bm{v}\in X} q^{Q(\bm{v})}.
\]
This theta function is what is known as a modular form.

A 
\begin{it}(holomorphic) modular form\end{it} of weight $k\in\mathbb{N}$ on $\Gamma\subseteq\operatorname{SL}_2(\Z)$ with character $\chi$ is a holomorphic function $f:\H\to\C$ satisfying the following properties:
\noindent

\noindent
\begin{enumerate}[leftmargin=*]
\item For $\gamma=\left(\begin{smallmatrix}a&b\\c&d\end{smallmatrix}\right)\in\Gamma$ we have 
\[
f|_{k}\gamma=\chi(d)f,
\]
where 
\[
f|_{k}\gamma(\tau):=(c\tau+d)^{-k}f\left(\frac{a\tau+b}{c\tau+d}\right).
\]
\item For every $\gamma\in\SL_2(\Z)$, the limit
\[
\lim_{\tau\to i\infty} f|_{k}\gamma(\tau)
\]
exists.
\end{enumerate}
We let $M_{k}\left(\Gamma,\chi\right)$ denote the space of modular forms of weight $k$ on $\Gamma$ with character $\chi$. If the limit in part (2) furthermore vanishes for every $\gamma\in\SL_2(\Z)$, then we call $f$ a \begin{it}cusp form\end{it}, and we let $S_{k}(\Gamma,\chi)$ denote the subspace of cusp forms.

Writing $\tau=u+iv$, there is a natural inner product 
\[
\left<f,g\right>:=\frac{1}{\left[\SL_2(\Z):\Gamma\right]} \int_{\Gamma\backslash\H} f(\tau)\overline{g(\tau)} v^k \frac{du dv}{v^2}
\]
defined on the subspace of cusp forms. This inner product is known as the \begin{it}Petersson inner product\end{it} and the induced \begin{it}Petersson norm\end{it} $\|f\|^2:=\left<f,f\right>$ is positive-definite on the space of cusp forms. The inner product between a holomorphic modular form $f$ and a cusp form $g$ also exists, and we define the \begin{it}Eisenstein series subspace\end{it} to be the subspace of holomorphic modular forms which are orthogonal to all cusp forms; note that the only cusp form which is also in the Eisenstein series subspace is the zreo function because of the positive-definite property. 

One can uniquely split $f\in M_k(\Gamma,\chi)$ as 
\[
f=E+g
\]
where $E$ is an Eisenstein series and $g\in S_k(\Gamma,\chi)$ is a cusp form. It is then natural to investigate the contributions from the Eisenstein series and cuspidal components for $f=\Theta_X$. For a more detailed introduction to modular forms, see \cite{OnoBook}.

\begin{extractsection}
\section{Eisenstein series part}\label{sec:Eisenstein}
\end{extractsection}
Based on \cite[Theorem 1.5]{ShimuraInhomogeneous}, we use formulas from \cite[Theorems 4.2 and 4.6]{KaneYang} to compute certain local densities whose product give the Fourier coefficients of the Eisenstein series part of the theta function formed by the generating function of the representations counted in Lemma \ref{lem:Lvd+vdj}. 

We first recall the setup from \cite{KaneYang} in our setting. Suppose that there is a quadratic lattice $\mathcal{L}=\bigoplus_{j=1}^{\ell}\Z \bm{v_j}$ with associated quadratic form 
\[
\mathcal{Q}\left(\sum_{j=1}^{\ell}n_j\bm{v_j}\right):=\sum_{j=1}^{\ell} b_j n_j^2.
\]
 Based on Lemma \ref{lem:Lvd+vdj}, we take $\mathcal{L}=L_{\bm{d}}$, $b_j=a_j(m-2)^2d_j^2$, and $\bm{\nu}=\bm{\nu}_{\bm{d}}$ (from \eqref{eqn:nuddef}), and set 
\[
X=X_{\bm{d}}:=L_{\bm{d}}+\bm{\nu}_{\bm{d}}.
\]
Consider the bilinear form 
\begin{equation}\label{eqn:bilineardef}
\mathcal{B}(\bm{x},\bm{y}):=\mathcal{Q}(\bm{x}+\bm{y})-\mathcal{Q}(\bm{x})-\mathcal{Q}(\bm{y})
\end{equation}
and set $B(\bm{x},\bm{y}):=\frac{1}{2}\mathcal{B}(\bm{x},\bm{y})$ so that
\[
\mathcal{Q}(\bm{x})=\frac{1}{2} \mathcal{B}(\bm{x},\bm{x})=B(\bm{x},\bm{x}).
\]
We then 
\rm
 let $L_X$ be any lattice satisfying $X\subseteq L_X$, 
\begin{equation}\label{eqn:LXPoundDef}
L_X^{\#}:=\left\{\bm{v}\in \Q L_X: 2B(\bm{v},\bm{x})\in \Z\ \forall \bm{x}\in L_X\right\},
\end{equation}
 and $\beta_{p}(h;X)$ are so-called local densities, which are computed by plugging \cite[Theorem 4.2 and Theorem 4.6]{KaneYang} into \cite[(4.1)]{KaneYang}. Since $L_{\bm{d}}\subseteq L$ (as every basis element $\bm{v_{d,j}}\in L$; see \eqref{eqn:vdjdef}) and $\bm{\nu}_{\bm{d}}\in \frac{1}{2} L$ (see \eqref{eqn:nuddef}), we can uniformly take $L_X:=\frac{1}{2}L$ in our setting. The Eisenstein series part of the theta function $\Theta_X$ is given in terms of these quantities in \cite[(3.2)]{KaneYang}, which for $\ell>2$ we recall in the following lemma. 
\begin{lemma}\label{lem:KaneYangProduct}
 For $\ell>2$ and 
\[
h=2n(m-2)+\left(\frac{m-4}{2}\right)^2\sum_{j=1}^{\ell} a_j,
\]
 the $h$-th Fourier coefficient of the Eisenstein series part of $\Theta_X$ is
\[
\frac{(2\pi)^{\frac{\ell}{2}} h^{\frac{\ell}{2}-1}}{\sqrt{\left[L_{X}^{\#}:L_X\right]}\Gamma\left(\frac{\ell}{2}\right)} \prod_{p} \beta_p(h;X).
\]
\end{lemma}
\begin{remarks}
\noindent

\noindent
\begin{enumerate}[leftmargin=*,label=\rm(\arabic*)]
\item  Note that if \eqref{eqn:summgonal} is solvable for all $n\in\N$, then $\ell>2$.
\begin{extradetails}
\begin{claim}\label{claim:universaldim>2}
\end{claim}
\end{extradetails}
\item Note that if we have a solution to \eqref{eqn:summgonal} satisfying the gcd conditions \eqref{eqn:gcdcondition} for those $n_j\neq 0$, and $(a_1,\dots,a_{\ell})$ is the first $\ell$ terms in $(a_{1},\dots, a_{\ell},a_{\ell+1},\dots,a_{L})$, then by taking $n_j=0$ for $j>\ell$ we get a solution to \eqref{eqn:summgonal} with $(a_{1},\dots, a_{\ell},a_{\ell+1},\dots,a_{L})$ as well. We may hence restrict to small $\ell$ and determine the choices of $n$ for which such a solution exists/no such solution exists, filling in the missing (generally small) $n$ by elementary means in higher dimensions. We may therefore restrict to the case $\ell\in\{4,6\}$. 
\end{enumerate}
\end{remarks}
We first compute 
\[
\left[L_X^{\#}:L_X\right].
\]
Assume, without loss of generality, that $\bm{e_j}$ are the standard basis elements, so that $L=\Z^{\ell}$. Using the definition \eqref{eqn:LXPoundDef} and $\mathcal{Q}(\bm{x})= \sum_{j=1}^{\ell} b_j x_j^2$ with $b_j=a_j(m-2)^2d_j^2$, we find that 
\[
L_X^{\#}=\left\{\bm{v}\in\Q L_X: v_j\in \frac{1}{a_j(m-2)^2d_j^2}\Z\right\}.
\]
\begin{extradetails}
\begin{claim}\label{claim:LXPoundEval}
\end{claim}
\end{extradetails}
Since $L_X=\frac{1}{2}\Z^{\ell}$, we see that 
\begin{equation}\label{eqn:LXindex}
\left[L_X^{\#}:L_X\right]=\prod_{j=1}^{\ell}\left(2a_j(m-2)^2d_j^2\right).
\end{equation}
\begin{extradetails}
\begin{claim}\label{claim:LXPoundIndex}
\end{claim}
\end{extradetails}
\rm

We next evaluate $\beta_{p}(n;X)$. We consider \cite[Subsection 4.1]{KaneYang} with $X=X_{\bm{d}}:=L_{\bm{d}}+{\bm{\nu}}_{\bm{d}}$. By \eqref{eqn:QLd+vdj}, the Gram matrix of $L_{\bm{d}}$ with respect to the basis elements $\bm{v_{d,j}}$ (these basis elements were denoted by $\nu_j$ in \cite{KaneYang}) is the diagonal matrix whose $(j,j)$-th component is $a_j(m-2)^2d_j^2$. Moreover, by the definition \eqref{eqn:nuddef}, we have
\[
\bm{\nu}_{\bm{d}}=\frac{4-m}{2(m-2)}\sum_{j=1}^{\ell}\frac{1}{d_j}{\bm{v_{d,j}}},
\]
so 
\[
s_j:=\frac{4-m}{2(m-2)d_j}\in\Q
\]
satisfy the condition given directly before \cite[(4.1)]{KaneYang} i.e., we have ${\bm{\nu}}_{\bm{d}}=\sum_{j=1}^{\ell}s_j \bm{v_{d,j}}$.

Set
\begin{align}
\label{eqn:bjdef} b_j&:=a_jd_j^2(m-2)^2,\\
\label{eqn:cjdef}c_j&:=4a_jd_j(4-m)(m-2),\\
\nonumber \phi({n})&:=\sum_{j=1}^{\ell} b_jx_j^2+c_j x_j.
\end{align}
We write $b_j=u_j p^{\mu_j}\text{ and } c_j=v_j p^{\nu_j}$ with $p\nmid u_jv_j$.

By \cite[(4.1)]{KaneYang} (see also \cite[(4.4)]{KaneYang}), we have 
\[
\beta_{p}(h;X)=p^{-\ord_p\left(\left[L_X:L_{\bm{d}}\right]\right)} b_{p}\left(h,\lambda_{\textbf{d}},0\right),
\]
where $b_{p}(h,\lambda_{\bm{d}},0)$ is the integral defined below \cite[(4.2)]{KaneYang} (with the notation $I_{p}(2(m-2)n;\phi)$ there) and evaluated in \cite[Theorem 4.2]{KaneYang}. Here $\lambda_{\bm{d}}$ is the characteristic function for the shifted lattice $X_{\bm{d}}$. Following \cite[Theorem 4.2]{KaneYang}, if $h\neq 0$, we write $h=u p^r$ with $p\nmid u$ (we set $r:=\infty$ if $h=0$), 
\begin{align}
\label{eqn:tjdef}t_j&:=\min\left\{\ord_p(b_j),\ord_{p}(c_j)\right\}=\min\left\{\mu_j,\nu_j\right\},\\
\label{eqn:Dpdef}D_p&:=\left\{1\leq j\leq \ell: \mu_j>\nu_j\right\},\\
\nonumber N_p&:=\left\{1\leq j\leq \ell:  \mu_j\leq \nu_j\right\},\\
\label{eqn:Tdef}T&:=\min\left\{t_j: j\in D_p\right\}.
\end{align}
If $D_p=\emptyset$, then we set $T:=\infty$. Setting $\varepsilon_d:=1$ if $d\equiv 1\pmod{4}$ and $\varepsilon_d:=i$ if $p\equiv 3\pmod{4}$ as usual, we also define 
\begin{align}
\label{eqn:Lptdef} \mathcal{L}_{p}(t)&:=\left\{ j\in N_p:t_j-t<0\text{ and }t_j-t\text{ is odd}\right\},\\
\nonumber \ell_{p}(t)&:=\#\mathcal{L}_p(t),\\
\label{eqn:deltaptdef} \delta_p(t)&:=\varepsilon_p^{3\ell_p(t)}\prod_{j\in \mathcal{L}_p(t)} \left(\frac{u_j}{p}\right),\\
\label{eqn:tauptdef}\tau_p(t)&:=t+\sum_{\substack{j\in N_p\\ t_j<t}} \frac{t_j-t}{2},\\
\label{eqn:omegapdef}\omega_p&:=\begin{cases} 0&\text{if }r\geq T,\\ -\frac{1}{p}&\text{if }r<T\text{ and $\ell_p(r+1)$ is even},\\
 \varepsilon_p\left(\frac{u}{p}\right)\frac{1}{\sqrt{p}}&\text{if }r<T\text{ and $\ell_p(r+1)$ is odd}.
\end{cases}
\end{align}
 By \cite[Theorem 4.2]{KaneYang}, we have 
\begin{equation}\label{eqn:Ipphi}
b_{p}\left(h,\lambda_{\textbf{d}},0\right)=1+\left(1-\frac{1}{p}\right)\sum_{\substack{1\leq t\leq \min\{T,r\}\\ \ell_p(t)\text{ even}}} \delta_p(t) p^{\tau_p(t)}+\delta_p(r+1)\omega_p p^{\tau_p(r+1)}.
\end{equation}
Here we let $\min\{\infty,r\}=r$. 

We define $b_j$ as in \eqref{eqn:bjdef} and $c_j$ as in \eqref{eqn:cjdef} and write $b_j=u_j p^{\mu_j}$ and $c_j=v_jp^{\nu_j}$ with $p\nmid u_jv_j$. For $d_j$ squarefree, counting the powers of $p$ in \eqref{eqn:bjdef} and \eqref{eqn:cjdef} yields
\begin{align}
\label{eqn:mujeval}\mu_j&=\ord_p(b_j)=\ord_p(a_j)+2\delta_{p\mid d_j}+2\ord_p(m-2)\\
\label{eqn:nujeval}\nu_j&=\ord_p(c_j)=\ord_p\left(a_j\right)+\delta_{p\mid d_j}+\ord_p(m-4)+\ord_p(m-2)+2\delta_{p=2}.
\end{align}
In particular, for $p\nmid 2(m-2)(m-4)\prod_{j=1}^{\ell}a_j$ we have (see the definition in \eqref{eqn:Dpdef})
\[
p\in D_p\Leftrightarrow p\mid d_j.
\]
We furthermore have (see the definition in \eqref{eqn:tjdef})
\[
t_j=\min\left\{\mu_j,\nu_j\right\}=\begin{cases}0&\text{if }p\nmid d_j,\\ 1&\text{if }p\mid d_j,\end{cases}
\]
and (see the definition in \eqref{eqn:Tdef})
\[
T=\begin{cases}1&\text{if }p\mid \prod_{j=1}^{\ell} d_j,\\ \infty&\text{if }p\nmid \prod_{j=1}^{\ell} d_j.\end{cases}
\]
For ease of notation, we also set 
\[
\alpha=\alpha_p:=\varepsilon_p^{3\ell} \prod_{j=1}^{\ell} \left(\frac{u_j}{p}\right).
\]
We first evaluate $b_{p}\left(h,\lambda_{\textbf{d}},0\right)$ in our case for primes $p\nmid 2(m-2)(m-4)\prod_{j=1}^{\ell}a_jd_j$. 
\begin{lemma}\label{lem:localdensityrelprime}
Suppose that $p$ is a prime satisfying $p\nmid 2(m-2)(m-4)\prod_{j=1}^{\ell}a_jd_j$. Then 
\begin{multline*}
b_{p}\left(h,\lambda_{\textbf{d}},0\right)=1+\left(1-\frac{1}{p}+\alpha p^{\frac{2-\ell}{2}}-\alpha p^{-\frac{\ell}{2}}\right)\sum_{1\leq t\leq \left\lfloor\frac{r-1}{2}\right\rfloor} p^{(2-\ell)t} + \alpha p^{\frac{2-\ell}{2}}\\
+\delta_{2\mid r}\left(1-\frac{1}{p}\right)\alpha p^{(2-\ell)\frac{r}{2}}-\delta_{2\nmid r} p^{(2-\ell)\frac{r+1}{2}-1}-\delta_{2\mid r} \alpha p^{(2-\ell)\frac{r+1}{2}-1}.
\end{multline*}
\end{lemma}
\begin{proof}
Since $p\nmid \prod_{j=1}^{\ell}d_j$, we have $N_p=\{1\leq j\leq\ell\}$, $t_j=0$ for all $j$, and $t_j-t=-t$, so (see \eqref{eqn:Lptdef})
\[
\mathcal{L}_p(t)=
\begin{cases}
\emptyset&\text{if $t$ is even},\\
\{1\leq j\leq\ell\}&\text{if $t$ is odd}.
\end{cases}
\]
Plugging this into \eqref{eqn:Ipphi}, we obtain
\begin{multline}\label{eqn:IpDpempty}
b_{p}\left(h,\lambda_{\textbf{d}},0\right)=1+\left(1-\frac{1}{p}\right)\sum_{\substack{1\leq t\leq r\\t \text{ even}}} \delta_p(t) p^{\tau_p(t)}\\
+\delta_{2\mid \ell}\left(1-\frac{1}{p}\right)
\sum_{\substack{1\leq t\leq r\\t \text{ odd}}} \delta_p(t) p^{\tau_p(t)}+\delta_p(r+1)\omega_p p^{\tau_p(r+1)}.
\end{multline}
For $t\in\N$ we then evaluate (see \eqref{eqn:deltaptdef})
\begin{equation}\label{eqn:deltaevalnmid}
\delta_{p}(t)=
\begin{cases}
1&\text{if $t$ is even},\\
\varepsilon_{p}^{3\ell} \prod_{j=1}^{\ell} \left(\frac{u_j}{p}\right)&\text{if $t$ is odd},
\end{cases}
\end{equation}
and (see \eqref{eqn:tauptdef}, plugging in $N_p=\{1\leq j\leq\ell\}$ and $t_j=0$ for all $j$)
\begin{equation}\label{eqn:tauevalnmid}
\tau_p(t)=t-\frac{\ell t}{2}=\frac{2-\ell}{2} t.
\end{equation}
We may then simplify \eqref{eqn:IpDpempty} as 
\begin{multline}\label{eqn:IpDpempty-2}
b_{p}\left(h,\lambda_{\textbf{d}},0\right)=1+\left(1-\frac{1}{p}\right)\sum_{\substack{1\leq t\leq r\\t \text{ even}}} p^{(2-\ell)\frac{t}{2}}+\delta_{2\mid \ell}\varepsilon_{p}^{3\ell} \prod_{j=1}^{\ell} \left(\frac{u_j}{p}\right)\left(1-\frac{1}{p}\right)p^{\frac{2-\ell}{2}}\sum_{\substack{1\leq t\leq r\\t \text{ odd}}}  p^{(2-\ell)\frac{t-1}{2}}\\
+\delta_{2\nmid r}\omega_p p^{(2-\ell)\frac{r+1}{2}}+\delta_{2\mid r}\varepsilon_p^{3\ell}\prod_{j=1}^{\ell}\left(\frac{u_{j}}{p}\right) \omega_p p^{(2-\ell)\frac{r+1}{2}}.
\end{multline}
If $\ell$ is odd, then the second sum in \eqref{eqn:IpDpempty-2} vanishes and (see \eqref{eqn:omegapdef})
\[
\omega_{p}=\begin{cases}
-\frac{1}{p}&\text{if $r$ is odd},\\
\varepsilon_p\left(\frac{u}{p}\right)\frac{1}{\sqrt{p}}&\text{if $r$ is even}.
\end{cases}
\]
 Plugging this, \eqref{eqn:deltaevalnmid}, and \eqref{eqn:tauevalnmid} into \eqref{eqn:IpDpempty-2} for $\ell$ odd then yields (making the change of variables $t\mapsto 2t$ in the first sum)
\begin{align*}
\nonumber b_{p}&\left(h,\lambda_{\textbf{d}},0\right)=1+\left(1-\frac{1}{p}\right)\sum_{1\leq t\leq \left\lfloor\frac{r}{2}\right\rfloor} p^{(2-\ell)t}-\delta_{2\nmid r}\alpha\frac{1}{p} p^{(2-\ell)\frac{r+1}{2}} + \delta_{2\mid r} \alpha\varepsilon_p\left(\frac{u}{p}\right)\frac{1}{\sqrt{p}} p^{(2-\ell)\frac{r+1}{2}}\\
\nonumber &=1+\left(1-\frac{1}{p}\right)\left(\frac{1-p^{(2-\ell)\left\lfloor\frac{r+2}{2}\right\rfloor}}{1-p^{2-\ell}}-1\right)-\delta_{2\nmid r}\frac{1}{p} p^{(2-\ell)\frac{r+1}{2}} + \delta_{2\mid r} \alpha \varepsilon_p\left(\frac{u}{p}\right)\frac{1}{\sqrt{p}} p^{(2-\ell)\frac{r+1}{2}}\\
\nonumber &=1+\left(1-\frac{1}{p}\right)\frac{p^{2-\ell}-p^{(2-\ell)\left\lfloor\frac{r+2}{2}\right\rfloor}}{1-p^{2-\ell}}-\delta_{2\nmid r}\frac{1}{p} p^{(2-\ell)\frac{r+1}{2}} + \delta_{2\mid r} \alpha \varepsilon_p\left(\frac{u}{p}\right)\frac{1}{\sqrt{p}} p^{(2-\ell)\frac{r+1}{2}}\\
\nonumber&=1+\left(1-\frac{1}{p}\right)\frac{1-p^{(2-\ell)\left\lfloor\frac{r}{2}\right\rfloor}}{p^{\ell-2}-1}-\delta_{2\nmid r}\frac{1}{p} p^{(2-\ell)\frac{r+1}{2}} + \delta_{2\mid r} \alpha\varepsilon_p\left(\frac{u}{p}\right)\frac{1}{\sqrt{p}} p^{(2-\ell)\frac{r+1}{2}}.
\end{align*}
For $\ell$ even, we make the change of variables $t\mapsto 2t$ in the first sum in \eqref{eqn:IpDpempty-2}  and the change of variables $t\mapsto 2t+1$ in the second sum in \eqref{eqn:IpDpempty-2}. Hence, for $\ell$ even, \eqref{eqn:IpDpempty-2}  simplifies as (note that $\varepsilon_p^{3\ell}=(\frac{-1}{p})$)
\begin{align*}
b_{p}\left(h,\lambda_{\textbf{d}},0\right)&=1+\left(1-\frac{1}{p}\right)\sum_{1\leq t\leq \left\lfloor\frac{r}{2}\right\rfloor} p^{(2-\ell)t}+\alpha \left(1-\frac{1}{p}\right)p^{\frac{2-\ell}{2}}\sum_{0\leq t\leq \left\lfloor\frac{r-1}{2}\right\rfloor}p^{(2-\ell)t}\\
&\hspace{5cm}+\delta_{2\nmid r} \omega_p p^{(2-\ell)\frac{r+1}{2}}+\delta_{2\mid r} \alpha \omega_p p^{(2-\ell)\frac{r+1}{2}}.
\end{align*}
Since $\ell$ is even, $\ell_{p}(r+1)\in\{0,\ell\}$ is even, so $\omega_p=-\frac{1}{p}$ and this simplifies as claimed.
\end{proof}

Next suppose that $p\nmid 2(m-2)(m-4)\prod_{j=1}^{\ell}a_j$ and $D_p\neq\emptyset$. In this case, we have $T=1$. Since $t_j=0$ for every $j\in N_p$, for $t\in\N$ we have  (see \eqref{eqn:Lptdef})
\[
\mathcal{L}_p(t)=
\begin{cases}
N_p&\text{if $t$ is odd},\\
\emptyset&\text{if $t$ is even}.
\end{cases}
\]
   If $r=0$, then $\min\{r,T\}=0$, so the sum on the right-hand side of \eqref{eqn:Ipphi} is empty. Hence the right-hand side of \eqref{eqn:Ipphi} becomes
\[
1+\delta_p(1)\omega_p p^{\tau(1)}.
\]
As $r=0<1=T$, we have (see \eqref{eqn:omegadef})
\[
\omega_p=\begin{cases} -\frac{1}{p}&\text{if $\# N_p$ is even},\\
 \varepsilon_p\left(\frac{u}{p}\right)\frac{1}{\sqrt{p}}&\text{if $\# N_p$ is odd}.
\end{cases}
\]
Setting 
\[
\alpha_p:=\prod_{p\in N_p}\left(\frac{u_j}{p}\right),  
\]
we also have (see \eqref{eqn:deltaptdef}) 
\[
\delta_p(1)=\varepsilon_p^{3\#N_p}\alpha_p
\]
and (see \eqref{eqn:tauptdef})
\[
\tau_p(1)=1-\frac{\#N_p}{2}.
\]
Thus (using $\varepsilon_p^2=(\frac{-1}{p})$) the right-hand side of \eqref{eqn:Ipphi} equals
\[
1+\delta_p(1)\omega_p p^{\tau(1)}=1-\delta_{2\mid \# N_p}\alpha_p\left(\frac{-1}{p}\right)^{\frac{\#N_p}{2}} p^{-\frac{\#N_p}{2}} -\delta_{2\nmid \# N_p}\alpha_p\left(\frac{-1}{p}\right)^{\frac{3\#N_p+1}{2}} \left(\frac{u}{p}\right)p^{\frac{1-\#N_p}{2}}. 
\]
Finally, for $r\geq 1$, the sum over $t$ in \eqref{eqn:Ipphi} has at most one term (which occurs if and only if $\#\mathcal{L}_p=\#N_p$ is even) and $r\geq 1=T$ implies that $\omega_p=0$ by \eqref{eqn:omegapdef}, so the final term vanishes. Hence the right-hand side of \eqref{eqn:Ipphi} becomes 
\[
1+\left(1-\frac{1}{p}\right)\delta_{2\mid \#N_p} \delta_p(1)p^{\tau_p(1)}=1+\left(1-\frac{1}{p}\right) \delta_{2\mid \#N_p} \left(\frac{-1}{p}\right)^{\frac{\#N_p}{2}}\alpha_p p^{1-\frac{\#N_p}{2}}.
\]

By a direct computation using \cite[Theorem 4.2]{KaneYang} (see \cite{Ng} for further details), we have the following results.\footnote{For local densities at $p\mid (m-4),$ one can refer \cite[Lemma 4.6]{Kanebanerjee}. Therefore we assume $p\nmid (m-4)$ below.}
\begin{lemma}\label{lem:localden} Suppose $p$ is an odd prime.
    \begin{enumerate}[leftmargin=*,label=\rm(\alph*)]
        \item If $p\mid (m-2),$ then $$b_{p}\left(h,\lambda_{\textbf{d}},0\right)=  
            \begin{cases}
              p^{T} & \text{if $r\geq T$},\\
              0 & \text{if $r< T$}.
            \end{cases}$$
        \item If $p\nmid (m-2)(m-4)$ and $p\mid\prod_{j=1}^4a_j,$ then for $\alpha=(0,0,0,1)$
             $$b_{p}\left(h,\lambda_{\textbf{d}},0\right)=  
            \begin{cases}
			1+\epsilon_p^2(\frac{u_1u_2u_3u}{p})p^{-1} & \text{if $r=0$},\\
            1+(\frac{u_4u}{p})p^{-2} & \text{if $r=1$},\\
            1+\epsilon_p^2(\frac{u_1u_2u_3u}{p})p^{-r-1} &\text{if $r\geq 2$ and even},\\
            1+(\frac{u_4u}{p})p^{-r-1} &\text{if $r\geq 3$ and odd}.
            \end{cases} $$    
        \item If $p\nmid (m-2)(m-4)$ and $p\mid\prod_{j=1}^4a_j,$ then for $\alpha=(0,0,0,3)$
             $$b_{p}\left(h,\lambda_{\textbf{d}},0\right)=  
            \begin{cases}
			1+\epsilon_p^2(\frac{u_1u_2u_3u}{p})p^{-1} & \text{if $r=0$},\\
            1-p^{-2} & \text{if $r=1$},\\
 
            1+p^{-1}-p^{-2}+\epsilon_p^2(\frac{u_1u_2u_3u}{p})p^{-2} &\text{if }r=2,\\
            1+p^{-1}-p^{-2} &\text{if $r\geq 3$}.
            \end{cases} $$
        \item If $p\nmid (m-2)(m-4)$ and $p\mid\prod_{j=1}^4a_j,$ then for $\alpha=(0,0,1,2)$
            $$b_{p}\left(h,\lambda_{\textbf{d}},0\right)=  
            \begin{cases}
        			1-\epsilon_p^2(\frac{u_1u_2}{p})p^{-1} & \text{if $r=0$},\\
                    1+(1-p^{-1})\epsilon_p^2(\frac{u_1u_2}{p}) & \text{if $r=1$},\\
                    1+(1-p^{-1})\epsilon_p^2(\frac{u_1u_2}{p}) &\text{if $r\geq 2$ and even},\\
                    1+(1-p^{-1})\epsilon_p^2(\frac{u_1u_2}{p}) &\text{if $r\geq 3$ and odd}.
            \end{cases} $$
        \item If $p\nmid (m-2)(m-4)$ and $p\mid\prod_{j=1}^4a_j,$ then for $\alpha=(0,0,2,3)$
            $$b_{p}\left(h,\lambda_{\textbf{d}},0\right)=  
            \begin{cases}
        			1-\epsilon_p^2(\frac{u_1u_2}{p})p^{-1} & \text{if $r=0$},\\
                    1+(1-p^{-1})\epsilon_p^2(\frac{u_1u_2}{p}) & \text{if $r=1$},\\
                    1+(1-p^{-1})\epsilon_p^2(\frac{u_1u_2}{p}) &\text{if $r\geq 2$ and even},\\
                    1+(1-p^{-1})\epsilon_p^2(\frac{u_1u_2}{p}) &\text{if $r\geq 3$ and odd}.
            \end{cases} $$
        \item If $p\nmid (m-2)(m-4)$ and $p\mid\prod_{j=1}^4a_j,$ then for $\alpha=(0,1,2,2)$
            $$b_{p}\left(h,\lambda_{\textbf{d}},0\right)=  
            \begin{cases}
        			1+(\frac{u_1u_n}{p}) & \text{if $r=0$},\\
                    1 & \text{if $r\geq1$}.
            \end{cases} $$
        \item If $p\nmid (m-2)(m-4)$ and $p\mid\prod_{j=1}^4a_j,$ then for $\alpha=(0,2,2,3)$
            $$b_{p}\left(h,\lambda_{\textbf{d}},0\right)=  
            \begin{cases}
        			1+(\frac{u_1u_n}{p}) & \text{if $r=0$},\\
                    1 & \text{if $r\geq1$}.
            \end{cases} $$
        \item If $p\nmid (m-2)(m-4)$ and $p\mid\prod_{j=1}^4a_j,$ then for $\alpha=(1,2,2,2)$
            $$b_{p}\left(h,\lambda_{\textbf{d}},0\right)=  
            \begin{cases}
        			0 & \text{if $r=0$},\\
                    p & \text{if $r\geq1$}.
            \end{cases} $$
        \item If $p\nmid (m-2)(m-4)$ and $p\mid\prod_{j=1}^4a_j,$ then for $\alpha=(2,2,2,3)$
            $$b_{p}\left(h,\lambda_{\textbf{d}},0\right)=  
            \begin{cases}
        			0 & \text{if $r=0$},\\
                    p & \text{if $r\geq1$}.
            \end{cases} $$
    \end{enumerate}
\end{lemma}
\begin{proof}
    We will demonstrate the proof of case (c). The proof of other cases are similar.
    \\ For $\alpha=(0,0,0,3),$ we have $N_p=\{1,2,3\}, t_1=t_2=t_3=0$ and $t_{\textbf{d}}=3.$ For $r=0,$ the middle sum is empty, thus
    $$b_{p}\left(h,\lambda_{\textbf{d}},0\right)=1+\delta_p(1)w_pp^{\tau_p(1)}=1+\epsilon_p^2(\frac{u_1u_2u_3u}{p})p^{-1}.$$
    For $r=1,$ the middle sum is empty as $\ell_p(3)$ is odd. Therefore 
    $$b_{p}\left(h,\lambda_{\textbf{d}},0\right)=1+\delta_p(2)w_pp^{\tau_p(2)}=1-p^{-2}.$$
    For $r=2$, we have
    \begin{align*}
        b_{p}\left(h,\lambda_{\textbf{d}},0\right)&=1+(1-\frac{1}{p})\delta_p(2)w_pp^{\tau_p(2)}+\delta_p(3)w_pp^{\tau_p(3)}
        \\&=1+p^{-1}-p^{-2}+\epsilon_p^2(\frac{u_1u_2u_3u}{p})p^{-2}.
    \end{align*}
    For $r\geq3$, we have $w_p=0$, and therefore
    \[
    b_{p}\left(h,\lambda_{\textbf{d}},0\right)=1+(1-\frac{1}{p})\delta_p(2)w_pp^{\tau_p(2)}=1+p^{-1}-p^{-2}.\qedhere
    \]
\end{proof}
Let $\mathbb{S}$ denote the set of squarefree positive integers.
For ease of notation, we write  $p^v\mid\mid \textbf{d}$ if $p^v\mid\mid \prod_{j=1}^4d_j.$ For $\textbf{d}\in \mathbb{S}^4$ with $p^v\mid\mid \textbf{d}$ we define $$\omega_v(p)=\omega_{v,\textbf{d}}(p):=\frac{b_{p}\left(h,\lambda_{\textbf{d}},0\right)}{b_{p}\left(h,\lambda_{\textbf{1}},0\right)} $$
and (we note that $w_v(p)$ depends on $\bm{d}$, but we suppress the dependence when it is clear)
$$\prod_{p\mid \textbf{d}}\beta_{X^{\textbf{d}},p}(h):=\frac{1}{d_1d_2d_3d_4}\prod_{p^v\mid\mid\textbf{d}}\omega_v(p).$$
We then define a function 
$$\frac{\Omega(p)}{p}:=\sum_{\substack{\textbf{d}\in \mathbb{S}^4 \\ \prod_{j=1}^4 d_j=p}}\frac{\omega_{1,\textbf{d}}(p)}{p}-\sum_{\substack{\textbf{d}\in \mathbb{S}^4 \\ \prod_{j=1}^4 d_j=p^2}}\frac{\omega_{2,\textbf{d}}(p)}{p^2}+\sum_{\substack{\textbf{d}\in \mathbb{S}^4 \\ \prod_{j=1}^4 d_j=p^3}}\frac{\omega_{3,\textbf{d}}(p)}{p^3}-\frac{\omega_{4,\textbf{d}}(p)}{p^4}.$$
By a direct computation using the above Lemmas, we have the following bound of $\frac{\Omega(p)}{p}$:

\begin{lemma}
Let $\textbf{a}\in \mathbb{N}^4,$ a prime $p\neq 2$ be given.
    \begin{enumerate}[label=(\alph*)]
        \item  Suppose $p\mid\mid \prod_{j=1}^4 a_j$ and $p\nmid (m-2)(m-4).$ Then by 
                $$\frac{\Omega(p)}{p}\leq 
            \begin{cases}
        			0.86 & \text{if $R=0, p=5$},\\
                    0.73 & \text{if $R=0, p\geq7$},\\
                    0.77 & \text{if $R=1,p=5$},\\
                    0.92, & \text{if $R=1, p\geq7$},\\
                    0.80 &\text{if $R\geq 2$ and even, $p=5$},\\
                    0.92 &\text{if $R\geq 2$ and even, $p\geq 7$},\\
                    0.78 &\text{if $R\geq 3$ and odd, $p=5$},\\
                    0.60 &\text{if $R\geq 3$ and odd, $p\geq 7$}.
            \end{cases} $$
        \item Suppose $p\mid\mid\prod_{j=1}^4a_j$ and $p\mid (m-4).$ Then
   $$\frac{\Omega(p)}{p}\leq
   \begin{cases}
       0.94 & \text{if $R=0, p=5$},\\
            0.69 & \text{if $R=0, p\geq 7$},\\
            0.77 &\text{if $R= 1, p=5$},\\
            0.94 &\text{if $R= 1, p\geq7$},\\
            0.52 &\text{if $R\geq 2$ and even, $p=5$},\\
            0.84 &\text{if $R\geq 2$ and even, $p\geq7$},\\
            0.52 &\text{if $R\geq 3$ and odd, $p=5$},\\
            0.80 &\text{if $R\geq 3$ and odd, $p\geq 7$}.
   \end{cases} $$
        \item Suppose $p\nmid\prod_{j=1}^4a_j,$ $p\nmid(m-2)(m-4).$ Then
        $$\frac{\Omega(p)}{p}\leq 
    \begin{cases}
			0.87 & \text{if $R=0, p=5$},\\
            0.59 & \text{if $R=0, p=7$},\\
            0.79 &\text{if $R= 1, p=5$},\\
            0.57 &\text{if $R= 1, p=7$},\\
            0.90 &\text{if $R\geq 2$ and even, $p=5$},\\
            0.64 &\text{if $R\geq 2$ and even, $p=7$},\\
            0.90 &\text{if $R\geq 3$ and odd, $p=5$},\\
            0.64 &\text{if $R\geq 3$ and odd, $p=7$}.
    \end{cases} $$
        \item Suppose $p\nmid\prod_{j=1}^4a_j$ and $p\mid (m-4).$ Then
   $$\frac{\Omega(p)}{p}\leq
   \begin{cases}
       0.90 & \text{if $R=0, p=5$},\\
            0.59 & \text{if $R=0, p\geq 7$},\\
            0.96 &\text{if $R= 1, p=5$},\\
            0.69 &\text{if $R= 1, p\geq7$},\\
            0.93 &\text{if $R\geq 2$ and even, $p=5$},\\
            0.71 &\text{if $R\geq 2$ and even, $p\geq7$},\\
            0.93 &\text{if $R\geq 3$ and odd, $p=5$},\\
            0.71 &\text{if $R\geq 3$ and odd, $p\geq 7$}.
   \end{cases} $$
        \item Suppose that $p\mid (m-2).$ Then
        $$\frac{\Omega(p)}{p}\leq 0.84.$$	
    \end{enumerate}
\end{lemma}
\begin{proof}
    We will compute the case for part (a) with the condition that $R=0.$ Other cases can be computed similarly.
    \\For $p=5,$ the maximum value of $\frac{\Omega(p)}{p}$ exists when $\textbf{a}=(1,2,2,5).$ Then by Lemma \ref{lem:localden}, we have 
    \begin{align*}
        \frac{\Omega(p)}{p} &\leq
        (\frac{1}{p}(1+p^{-1}+2(1+p^{-1})+1-p^{-1})-\frac{1}{p^{2}}(2(1+p^{-1})+1-p^{-1}+2)+\frac{2}{p^3})/(1-p^{-1})
        \\ &=0.86 
    \end{align*}
    For $p\geq 7,$ we can trivially bound $\frac{\Omega(p)}{p}$ by
    \begin{align*}
        \frac{\Omega(p)}{p} &\leq
        (4\cdot\frac{1}{p}(1+p^{-1})-3\cdot\frac{1}{p^2}(1-p^{-1})-3\cdot0+3\cdot\frac{2}{p^3})/(1-p^{-1})
        \\&\leq 0.73\qedhere
    \end{align*}
\end{proof}

\begin{extractsection}
\section{Sieve}\label{sec:sieve}
\end{extractsection}

We apply sieving theory to remove the representations that have
$p \mid d_j$ for $p \leq y$ with some $y$ depending on $n\in \mathbb{N}.$ Let $A_{\textbf{1}}$ be the set of solution $\textbf{x}\in \mathbb{Z}^4$ to $$\sum_{j=1}^4a_jP_m(x_j)=n.$$ 
Let $\mathbb{S}$ denote the set of squarefree positive integers. For $\textbf{d}\in \mathbb{S}^4$ with $\gcd(d_j,6)=1,$ we define
$$A_{\textbf{d}}:=\{\textbf{x}\in A_{\textbf{1}}:d_j\mid x_j\}.$$
Then we have $$r_{Q_{\textbf{a}\cdot\textbf{d}^2}}(n)=\#A_{\textbf{d}}.$$
Defining $$R(\textbf{d},n):=r_{Q_{\textbf{a}\cdot\textbf{d}^2}}(n)-a_{E_{Q_{\textbf{a}\cdot\textbf{d}^2}}}(n)$$ to be the coefficient of the cuspidal part of the theta function. Then we have the following proposition:
\begin{proposition}
    For $\textbf{d}\in \mathbb{N}^4$ with $\gcd(d_j,6)=1,$ we have 
    $$\#A_{\textbf{d}}=X\prod_{p\mid \textbf{d}}\beta_{X^{\textbf{d}},p}(h)+R(\textbf{d},h)$$
\end{proposition}
Setting 
\[
w_1(p):=\max\left\{\omega_{1,\bm{d}}(p):p\|\bm{d}\right\},
\]
we next bound the product of the reciprocals of $1-\frac{w_1(p)}{p}$  coming from precisely one component being divisible by $p$.
\begin{lemma}\label{lem:1-BetaBound}
Suppose that $\textbf{a}\in \mathbb{N}^4$  is only divisible by primes less than 7 and for each prime $5\leq p \leq 7,$ we have $\text{ord}_{p}\prod_{j=1}^4a_j \leq 1.$ Then for $w\geq3,$ we have
    $$\prod_{w<p<z}(1-\frac{w_1(p)}{p})^{-1}\leq 4\prod_{w<p<z}(1-\frac{1}{p})^{-2}.$$
\end{lemma}
\begin{proof}
    Bounding case by case, we find that for $p\nmid \prod_{j=1}^4a_j$, we have $b_p(h,\lambda_{p^{(1,0,0,0)}},0)\leq 1+\frac{1}{p}.$ Thus we have $$\frac{w_1(p)}{p}\leq \frac{1}{p}(1+\frac{1}{p})(\frac{p^3}{(p-1)^2(p+1)}=\frac{p}{(p-1)^2}.$$
    For $p\mid \prod_{j=1}^4a_j,$ we have by assumption that $5 \leq p \leq 7.$ A direct calculation for $5 \leq p \leq 7$ then shows that
    $$\prod_{5 \leq p \leq 7}(1-\frac{w_1(p)}{p})^{-1}\leq\prod_{5 \leq p \leq 7}(1-\frac{2p^2}{(p-1)^2(p+1)})^{-1}< 4.$$
    Therefore 
    \begin{multline*}
        \prod_{w < p < z}(1-\frac{w_1(p)}{p})^{-1} \leq \prod_{\substack{w < p < z \\ p\nmid \prod_{j=1}^4a_j}}(1-\frac{p}{(p-1)^2})^{-1}\prod_{5 \leq p \leq 7}(1-\frac{w_1(p)}{p})^{-1} 
        \\ \leq 4\prod_{\substack{w < p < z \\ p\nmid \prod_{j=1}^4a_j}}(1-\frac{5}{3p})^{-1} \leq 4\prod_{w<p<z}(1-\frac{1}{p})^{-2}. \qedhere
    \end{multline*}
\end{proof}

For a set $S$, we define $\chi_S(x)$ to be the characteristic function $\chi_S(x):=1$ if $x \in S$ and $\chi_S(x):= 0$ otherwise.
\begin{lemma}
For $1<w\leq z$ and $S\subset \mathbb{N},$ we have 
    $$\prod_{\max(w,5)\leq p <z}(1-\chi_{S}(p)\frac{w_1(p)}{p})^{-1}\leq 6(\frac{\log(z)}{\log(w)})(1+\frac{6}{\log(w)}).$$
\end{lemma}
\begin{proof}
    Since $$(1-\chi_{S}(p)\frac{w_1(p)}{p})^{-1}\leq (1-\frac{w_1(p)}{p})^{-1},$$
    it suffices to prove the claim for $S = \mathbb{N}.$
    Using Lemma \ref{lem:1-BetaBound} we bound,
    \begin{multline*}
        \prod_{\max(w,5)\leq p \leq z}(1-\frac{w_1(p)}{p})^{-1} \leq 4\prod_{\max(w,5)\leq p \leq z}(1-\frac{p}{(p-1)^2})^{-1} 
        \\\leq 6\prod_{\max(w,7)\leq p \leq z}\frac{p}{p-1}\prod_{\max(w,7)\leq p \leq z}(1-\frac{3}{p^2})^{-1}.
    \end{multline*}
    By \cite[Lemma 5.3]{Kanebanerjee}, we have 
    $$\prod_{\max(w,5)\leq p <z}(1-\chi_{S}(p)\frac{w_1(p)}{p})^{-1}\leq 6(\frac{\log(z)}{\log(w)})(1+\frac{6}{\log(w)}).$$
\end{proof}

For $\beta, D>0$ and the integer $d$ of the form $d=p_1\cdot\cdot\cdot p_r$ with $p_1 > \cdot\cdot\cdot > p_r$ with $p_j$ an odd prime.
        The Rosser weights $\lambda_d^{\pm}$ are defined as follows:
        \\Let $$y_m=y_m(D,\beta):=(\frac{D}{p_1\cdot\cdot\cdot p_m})^{\frac{1}{\beta}},$$
        then
        \begin{align*}
            \lambda_d^+=&\lambda_{d,D}^+(\beta):=
            \begin{cases}
                (-1)^r \quad \text{if $p_{2l+1}<y_{2l+1}(D,\beta) \quad\forall 0\leq l\leq \frac{r-1}{2}$},
                \\0 \quad\qquad\text{otherwise.}
            \end{cases}
            \\\lambda_d^-=&\lambda_{d,D}^-(\beta):=
                \begin{cases}
                (-1)^r \quad \text{if $p_{2l}<y_{2l}(D,\beta) \quad\forall 0\leq l\leq \frac{r}{2}$},
                \\0 \quad\qquad\text{otherwise.}
            \end{cases}
        \end{align*}

    Furthermore define 
    $$\Lambda_d^-:=4\lambda_d^--3\lambda_d^+.$$

As is standard, we consider $D$ and $\beta$ to be fixed throughout and omit these in the notation. For $\beta > 1,$ we define $a=a_{\beta}:=e\frac{\beta}{\beta-1}\log(\frac{\beta}{\beta-1}), r=r_{\beta}:=\frac{\log(1+\frac{6}{\log(5)})}{\log(\frac{\beta}{\beta-1})}$ and 
$$C_{\beta}(s):=2e^{r_{\beta}-1}(1+\frac{6}{\log(5)})\frac{a_{\beta}^{\lfloor s-\beta \rfloor+1}}{1-a_{\beta}}$$
\begin{lemma}
    Suppose that for $\textbf{a} \in \mathbb{N}^4 $ we have $p \nmid a_j$ for every $p \geq 11$ and for $5 \leq  p \leq 7$ we have
$\text{ord}_p\prod_{j=1}^4a_j \leq 1.$ Let a subset P of primes be given and set $S = S_P$ to be the set of all squarefree integers for which $d \in S$ if and only if all prime divisors of d are in P. Let $D > 0$ and $\beta \geq 5$ be
given and set $s :=\frac{\log(D)}{\log(z)}
. $Then for $s \geq \beta$ and $z \geq 5$ the following hold:

$$\sum_{d\mid P_{5}(z)}\lambda_d^{+}\chi_{S}\prod_{p\mid d}\beta_{X^{(d,1,1,1)},p}(h)\leq \prod_{5\leq p \leq z}(1-\chi_{S}(p)\frac{w_1(p)}{p})(1+C_{\beta}(s)),$$

$$\sum_{d\mid P_{5}(z)}\lambda_d^{-}\chi_{S}\prod_{p\mid d}\beta_{X^{(d,1,1,1)},p}(h)\geq \prod_{5\leq p \leq z}(1-\chi_{S}(p)\frac{w_1(p)}{p})(1-C_{\beta}(s)).$$
\end{lemma}
\begin{proof}
    The proof is the same as the proof of \cite[Lemma 5.4]{Kanebanerjee}.
\end{proof}

\begin{lemma}\label{lem:rosbound}
    Let $D > 0$ and $\beta \geq 5$ be given and set $s := \frac{\log(D)}{\log(z)}$
. Then for $s \geq \beta, z \geq 7,$ and
squarefree $\delta \in \mathbb{N}$ with $\gcd(\delta, 6) = 1$ the following hold:

$$\sum_{d\mid P_{5}(z)}\lambda_d^{+}\prod_{p\mid d}\beta_{X^{(d,1,1,1)},p}(h)\leq \mu(\delta)\prod_{p\mid \delta}\frac{\frac{w_1(p)}{p}}{1-\frac{w_1(p)}{p}}\prod_{5\leq p \leq z}(1-\frac{w_1(p)}{p})(1+C_{\beta}(s)),$$

$$\sum_{d\mid P_{5}(z)}\lambda_d^{-}\prod_{p\mid d}\beta_{X^{(d,1,1,1)},p}(h)\geq \mu(\delta)\prod_{p\mid \delta}\frac{\frac{w_1(p)}{p}}{1-\frac{w_1(p)}{p}}\prod_{5\leq p \leq z}(1-\frac{w_1(p)}{p})(1-C_{\beta}(s)).$$
\end{lemma}
\begin{proof}
    The proof is the same as the proof of \cite[Lemma 5.5]{Kanebanerjee}.
\end{proof}

\begin{lemma}\label{lem:Xbound}
    Suppose that 
    \begin{align*}
        \textbf{a}\in &\{(1,1,1,k): 1\leq k \leq 4\} \cup \{(1,1,2,k): 2\leq k \leq 5\} \cup\{(1,1,3,k): 3\leq k \leq 6\} \\&\cup \{(1,2,2,k): 2\leq k \leq 6\} \cup \{(1,2,3,k): 3\leq k \leq 7\} \cup \{(1,2,4,k): 4\leq k \leq 8\}.
    \end{align*} 
    \\Then we have 
    $$X=a_{E_{Q_{\textbf{a},1}}}(h)\geq \frac{9}{26000(m-2)^{3+10^{-6}}}h^{1-10^{-6}}.$$
\end{lemma}
\begin{proof}
    Note that we have 
    $$a_{E_{X}}(h)=\frac{(2\pi)^2h}{(16d_{L^{\textbf{d}}})^\frac{1}{2}\Gamma(2)L(2,\psi)}\cdot \prod_{p\mid e_1}\frac{b_p(h,\lambda_{\textbf{d}},0)}{(1-\psi (p)p^{-2})}\cdot\prod_{\substack{p\mid h\\{p \nmid e_1}}}\gamma_p(2).$$
    We first compute $\prod_{p\mid 2(m-2)}b_{p}(h,\lambda,0).$
    We have 
    \begin{align*}
        \prod_{p\mid 2(m-2)}b_{p}(h,\lambda_{\textbf{d}},0) &=b_{2}(h,\lambda_{\textbf{d}},0)\prod_{\substack{p\mid (m-2) \\ p \text{ odd}}}b_{p}(h,\lambda_{\textbf{d}},0)\\&= \frac{1}{2^{\text{ord}_{2}(m-2)+1}}\cdot\prod_{\substack{p\mid (m-2) \\ p \text{ odd}}}\frac{1}{p^{3\text{ord}(m-2)}}\\ &\geq \frac{1}{2(m-2)^3}.
    \end{align*}
    For $p\nmid 2(m-2),$ we have 
    \begin{align*}
        \prod_{\substack{p\mid e_1 \\ p\nmid 2(m-2)}}b_{p}(h,\lambda_{\textbf{d}},0) &=\prod_{\substack{p\mid \prod_{j=1}^4a_j \\ p\nmid 2(m-2)}}b_{p}(h,\lambda_{\textbf{d}},0)
        \\&\geq \frac{2}{3}\cdot\frac{4}{5}
        \\& =\frac{8}{15}
    \end{align*}
    Note that we have $$\prod_{\substack{p\mid h \\ p\nmid e_1e'}}\gamma_p(2) \geq \prod_{p\mid h}(1-\frac{1}{p}) \geq \frac{1}{20}h^{-10^{-6}},$$ and 
    $$L(2,\psi)^{-1}\geq \zeta(2)^{-1}=\frac{6}{\pi^2}.$$
    By bounding $\Delta_{\textbf{a}}\leq 8\times4\times 2\times 16=1024,$ we have
    $$\frac{4\pi^2}{\Delta_{\textbf{a}}^{\frac{1}{2}}L(2,\psi)}\prod_{p\mid e_1}\frac{1}{1-\psi(p)p^{-2}}\geq \frac{27}{1040}(m-2)^{-10^{-6}}.$$
    Therefore 
    $$a_{E_{X}}(h)\geq \frac{9}{26000(m-2)^{3+10^{-6}}}h^{1-10^{-6}}.$$
\end{proof}


For $w\in \mathbb{R}$, define
$$S(A_{\textbf{1}},z):=\#\{\textbf{x}\in A_{\textbf{1}}:\gcd(x_j,P_{w}(z))=1\}=\sum_{\textbf{x}\in A_{\textbf{1}}}\prod_{j=1}^4(\mu * \textbf{1})(\gcd(x_j,P_{w}(z))).$$
\\We then define
$$\sum(D,z):=\sum_{d_1\mid P_{5}(z)}\sum_{d_2\mid P_{5}(z)}\sum_{d_3\mid P_{5}(z)}\sum_{d_4\mid P_{5}(z)}\Lambda^{-}_{d_1}\lambda^{+}_{d_2}\lambda^{+}_{d_3}\lambda^{+}_{d_4}\prod_{p\mid \textbf{d}}\beta_{X^{\textbf{d}},p}(h),$$

$$\sideset{}{^{\prime}}{\sum}(D,z):=\sum_{d_1\mid P_{5}(z)}\sum_{d_2\mid P_{5}(z)}\sum_{d_3\mid P_{5}(z)}\sum_{d_4\mid P_{5}(z)}\lambda^{+}_{d_1}\lambda^{+}_{d_2}\lambda^{+}_{d_3}\lambda^{+}_{d_4}\prod_{p\mid \textbf{d}}\beta_{X^{\textbf{d}},p}(h).$$
We next obtain an upper and lower bound for $S(A_{\textbf{1}},z).$

\begin{lemma}\label{lem:uplowboundS}
For $w\geq 5,$we have    $$X\sum(D,z)-7\sum_{\substack{\textbf{d}\in\mathbb{Z}^4\\d_j\mid P_{w}(z)\\ \mid d_j\mid\leq \frac{D}{7^{\beta-1}}}}\mid R(\textbf{d},h)\mid \leq S(A,z) \leq X\sideset{}{^{\prime}}{\sum}(D,z)+\sum_{\substack{\textbf{d}\in\mathbb{Z}^4\\d_j\mid P_{w}(z)\\ \mid d_j\mid\leq \frac{D}{7^{\beta-1}}}}\mid R(\textbf{d},h)\mid$$
\end{lemma}
\begin{proof}
    See \cite[Lemma 5.7]{Kanebanerjee}.
\end{proof}

Define 
\begin{multline*}
    \sum_{\text{MT}}:=\prod_{5\leq p \leq z}(1-\frac{w_1(p)}{p})^4\sum_{d_{1,2}}\cdot\cdot\cdot\sum_{d_{3,4}}g((d_{i,j}))
    \\\times \sum_{l_{i,j}\mid \frac{P_{5}(z)}{d_{i,j}}}\mu(l_{1,2})\cdot\cdot\cdot \mu(l_{3,4})\prod_{j=1}^4\mu(\xi_j)\prod_{p\mid \xi_j}\frac{\frac{w_1(p)}{p}}{1-\frac{w_1(p)}{p}}
\end{multline*}

\begin{lemma}
    We have 
    $$\sum(D,z)\geq (1-7C_{\beta}(s))(1-C_{\beta}(s))^3\sum_{\text{MT}}.$$
\end{lemma}
\begin{proof}
    See \cite[Lemma 5.8]{Kanebanerjee}.
\end{proof}

\begin{lemma}
    We have 
    $$\sideset{}{^{\prime}}{\sum}(D,z)\leq (1+C_{\beta}(s))^4\sum_{\text{MT}}.$$
\end{lemma}
\begin{proof}
    See \cite[Lemma 5.9]{Kanebanerjee}.
\end{proof}

\begin{lemma}
We have 
    $$\sum_{MT}=\prod_{5\leq p \leq z}(1-\frac{\Omega(p)}{p}).$$
\end{lemma}
\begin{proof}
    See \cite[Lemma 5.10]{Kanebanerjee}.
\end{proof}

\begin{lemma}\label{lem:SumDzBound}
    Suppose that $\textbf{a} \in \mathbb{N}^4$ has at most one prime $p \geq 7$ dividing $\prod_{j=1}^4 a_j$ and moreover
that $p\mid\mid \prod_{j=1}^4a_j$ and $5 \leq p \leq 7.$ Take $\beta =10, D\geq 27,$ then 
$$\sum(D,z)\geq \frac{7}{10}\prod_{5\leq p \leq z}(1-\frac{\Omega(p)}{p}) \geq \frac{7}{50}\prod_{5\leq p \leq z}(1-\frac{4.93}{p}) \geq 0.68\prod_{p\leq z}(1-\frac{1}{p})^5.$$
\end{lemma}
\begin{proof}
A direct calculation shows that $$C_{\beta}(s)\leq \frac{1}{33}.$$
By bounding $\Omega(p)\leq 4.93$ for $p\nmid\prod_{j=1}^4a_j$, we have
$$\prod_{\substack{5\leq p \leq 7 \\ p\mid\mid \prod_{j=1}^4a_j}}\frac{1-\frac{\Omega(p)}{p}}{1-\frac{4.93}{p}}\geq \text{min}(\frac{1-\frac{4.7}{5}}{1-\frac{4.93}{5}},\frac{1-\frac{6.58}{7}}{1-\frac{4.93}{7}})\geq 0.2.$$
Therefore we have 
    \begin{align*}
        \sum(D,z)& \geq \frac{7}{10}\prod_{5\leq p \leq z}(1-\frac{\Omega(p)}{p})\\& \geq \frac{7}{50}\prod_{5\leq p \leq z}(1-\frac{4.93}{p}) \\&\geq 34\prod_{5\leq p \leq 139}\frac{(1-\frac{4.93}{p})}{(1-\frac{1}{p})^5}\prod_{p\leq z}(1-\frac{1}{p})^5
        \\& \geq 0.68\prod_{p\leq z}(1-\frac{1}{p})^5.\qedhere
    \end{align*}
\end{proof}

We next bound the cuspidal contribution to
obtain a bound for $S(A , z).$
\begin{lemma}\label{lem:CuspidalBound}
    For $\beta \geq 10,$ we have 
    $$\sum_{\substack{\textbf{d}\in\mathbb{Z}^4\\d_j\mid P_{w}(z)\\ \mid d_j\mid\leq \frac{D}{7^{\beta-1}}}}\mid R(\textbf{d},m)\mid \leq 2.04\times 10^{-64}(m-2)^{6+\frac{2}{10}+\frac{1}{100}+6\cdot10^{-6}}h^{\frac{17}{30}}D^{28.85}$$
\end{lemma}
\begin{proof}
    By \cite[Lemma 4.1]{Kanebanerjee}, we have 
    \begin{equation}
        \mid R(d,h)\mid \leq 4.58\times 10^{128}h^{\frac{17}{30}}\mid\mid f\mid\mid N_{\textbf{a},\textbf{d}^2}^{1+2\cdot 10^{-6}}\prod_{p\mid N_{\textbf{a},\textbf{d}^2}}(1+\frac{1}{p})^{\frac{1}{2}}\varphi(L),
    \end{equation}
    where $N_{\textbf{a},\textbf{d}^2}:=16(m-2)^2\text{lcm}(\textbf{a})\text{lcm}(\textbf{d})^2$ and $L= M_{\textbf{a},\textbf{d}^2}:=\frac{2(m-2)\text{lcm}(\textbf{d})}{\gcd(m-4,\text{lcm}(\textbf{d}))}.$ 
    \\We first bound $N_{\textbf{a},\textbf{d}^2} \leq 672(m-2)^2\text{lcm}(\textbf{d})^2$ and $M_{\textbf{a},\textbf{d}^2}\leq 2(m-2)\text{lcm}(\textbf{d}).$
    \\By \cite[(3.22)]{KamarajKane-Tomiyasu} we have
    $$\mid\mid f \mid\mid^2 \leq \frac{3^6\cdot 2}{\pi^4}\frac{(N_{\textbf{a},\textbf{d}^2})^2}{\prod_{p\mid N_{\textbf{a},\textbf{d}^2}}(1-p^{-2})}\sum_{\delta\mid N_{\textbf{a},\textbf{d}^2}}\hspace{-.1cm}\varphi(\frac{N_{\textbf{a},\textbf{d}^2}}{\delta})\varphi(\delta)\frac{N_{\textbf{a},\textbf{d}^2}}{\delta}(\frac{\gcd(M_{\textbf{a},\textbf{d}^2},\delta)}{M_{\textbf{a},\textbf{d}^2}})^4\times (\frac{27N_{\textbf{a},\textbf{d}^2}}{16\pi\prod_{j=1}^4\alpha_j}+16).$$
    Then by \cite[Lemma 2.3]{Kanebanerjee} and \cite[Lemma 2.4]{Kanebanerjee} we have 
    $$\sum_{\delta \mid N}\frac{N}{\delta}\varphi(\delta)\varphi(\frac{N}{\delta})\leq N\varphi(N)\sigma_{-1}(N) < 1.21\times 10^{121}\cdot N^{1+\frac{1}{10}+\frac{1}{100}}.$$
    Finally by bounding $\prod_{p\mid m}(1-p^{-2})\geq \frac{1}{20}m^{-10^{-6}},$ we have 
    \begin{align}
        \mid\mid f \mid\mid &\leq (3.86\times 10^{14}(N_{\textbf{a},\textbf{d}^2})^{3+\frac{1}{10}+\frac{1}{100}+10^{-6}}(\frac{27N_{\textbf{a},\textbf{d}^2}}{16\pi\prod_{j=1}^4\alpha_j}+16))^{\frac{1}{2}} \nonumber
        \\& \leq (3.86\times 10^{14}(672(m-2)^2\text{lcm}(\textbf{d})^2)^{3+\frac{1}{10}+\frac{1}{100}+10^{-6}}(361(m-2)^2\text{lcm}(\textbf{d})^2+16))^{\frac{1}{2}}
    \end{align}
    Plugging (4.2) into (4.1), we have
    \begin{multline*}
        \mid R(d,h)\mid \leq 2.25\times 10^{140}h^{\frac{17}{30}}((m-2)\text{lcm}(\textbf{d}))^{3+\frac{1}{10}+\frac{1}{100}+10^{-6}}(361(m-2)^2\text{lcm}(\textbf{d})^2+16)^{\frac{1}{2}}
        \\\cdot (672(m-2)^2\text{lcm}(\textbf{d})^2)^{1+2\cdot10^{-6}}\prod_{p\mid N_{\textbf{a},\textbf{d}^2}}(1+\frac{1}{p})^{\frac{1}{2}}\varphi(2(m-2)\text{lcm}(\textbf{d}))
    \end{multline*}
     Using \cite[Lemma 2.3]{Kanebanerjee} and \cite[Lemma 2.4]{Kanebanerjee} again, we have 
     $$\mid R(d,h)\mid \leq 2.29\times 10^{154}h^{\frac{17}{30}}((m-2)\text{lcm}(\textbf{d}))^{5+\frac{2}{10}+\frac{1}{100}+6\cdot10^{-6}}(361(m-2)^2\text{lcm}(\textbf{d})^2+16))^{\frac{1}{2}}.$$
     Thus 
     \begin{multline*}
         \sum_{\substack{\textbf{d}\in\mathbb{Z}^4\\d_j\mid P_{w}(z)\\ \mid d_j\mid\leq \frac{D}{7^{\beta-1}}}}\mid R(\textbf{d},h)\mid \leq 2.29\times 10^{154}h^{\frac{17}{30}}(m-2)^{5+\frac{2}{10}+\frac{1}{100}+6\cdot10^{-6}}
         \\\times\sum_{\substack{\textbf{d}\in\mathbb{Z}^4\\d_j\mid P_{w}(z)\\ \mid d_j\mid\leq \frac{D}{7^{\beta-1}}}}\text{lcm}(\textbf{d})^{5+\frac{2}{10}+\frac{1}{100}+6\cdot10^{-6}}(361(m-2)^2\text{lcm}(\textbf{d})^2+16))^{\frac{1}{2}}.
     \end{multline*}
    Now we bound $\text{lcm}(\textbf{d})\leq \prod_{j=1}^4d_j,$ the inner sum may be bounded against
    \begin{align*}
        \sum_{\substack{\textbf{d}\in\mathbb{Z}^4\\d_j\mid P_{w}(z)\\ \mid d_j\mid\leq \frac{D}{7^{\beta-1}}}}\text{lcm}(\textbf{d})^{5+\frac{2}{10}+\frac{1}{100}+6\cdot10^{-6}}&(361(m-2)^2\text{lcm}(\textbf{d})^2+16))^{\frac{1}{2}}  
        \\&\leq \sum_{\substack{\textbf{d}\in\mathbb{Z}^4\\d_j\mid P_{w}(z)\\ \mid d_j\mid\leq \frac{D}{7^{\beta-1}}}}\prod_{j=1}^4d_j^{6+\frac{2}{10}+\frac{1}{100}+6\cdot10^{-6}}(361(m-2)^2+16))^{\frac{1}{2}}
        \\&\leq (361(m-2)^2+16))^{\frac{1}{2}}\prod_{j=1}^4\sum_{d_j\leq \frac{D}{7^9}}d_j^{6+\frac{2}{10}+\frac{1}{100}+6\cdot10^{-6}}
        \\& \leq 20(m-2)(\frac{D}{7^9})^{4(7+\frac{2}{10}+\frac{1}{100}+6\cdot10^{-6})}
    \end{align*}
    Therefore we have 
\[
\sum_{\substack{\textbf{d}\in\mathbb{Z}^4\\d_j\mid P_{w}(z)\\ \mid d_j\mid\leq \frac{D}{7^{\beta-1}}}}\mid R(\textbf{d},m)\mid \leq 2.04\times 10^{-64}(m-2)^{6+\frac{2}{10}+\frac{1}{100}+6\cdot10^{-6}}h^{\frac{17}{30}}D^{28.85}.\qedhere
\]
\end{proof}
Combining together the previous lemmas yields a lower bound for $S(A,z)$.
\begin{lemma}\label{lem:SAzlower}
    By plugging in Lemma \ref{lem:Xbound}, Lemma \ref{lem:SumDzBound} and Lemma \ref{lem:CuspidalBound} into Lemma \ref{lem:uplowboundS}, we have 
\begin{multline*}
S(A,z)\geq \frac{2.35\times10^{-4}}{(m-2)^{3+10^{-6}}}h^{1-10^{-6}}\frac{e^{-5\gamma}}{(\log(z))^5}(1-\frac{1}{(\log(z)^2}))^5\\
-1.43\times 10^{-63}(m-2)^{6+\frac{2}{10}+\frac{1}{100}+6\cdot10^{-6}}h^{\frac{17}{30}}D^{28.85}.
\end{multline*}
\end{lemma}

\begin{extractsection}
\section{Proof of Theorem \ref{thm:main}}\label{sec:mainproof}
\end{extractsection}
We are now ready to put everything together and prove Theorem \ref{thm:main}.
\begin{proof}[Proof of Theorem \ref{thm:main}]
    By Lemma \ref{lem:SAzlower}, the number of representations $\sum_{j=1}^4a_jP_m(x_j)=n$ with $p\mid x_j \implies p\in \{2,3\}$ or $p\geq z$ is 
    \begin{multline*}
        S(A_{\textbf{1}},z) \geq \frac{2.35\times10^{-4}}{(m-2)^{3+10^{-6}}}h^{1-10^{-6}}\frac{e^{-5\gamma}}{(\log(z))^5}(1-\frac{1}{(\log(z)^2}))^5\\
    -1.43\times 10^{-63}(m-2)^{6+\frac{2}{10}+\frac{1}{100}+6\cdot10^{-6}}h^{\frac{17}{30}}D^{28.85}
    \end{multline*}
    with $D\geq z^{27}.$ We therefore choose $D=z^{27}.$ taking $z=\max(h^{\frac{1}{1800}},5),$ we have 
    \begin{align*}
      S(A_{\textbf{1}},z) \geq &\frac{2.35\times10^{-4}}{(m-2)^{3+10^{-6}}}h^{1-10^{-6}}\frac{1500^5e^{-5\gamma}}{(\log(h))^5}(1-\frac{1}{(\log(5)^2}))^5\\
      &\hspace{6cm}-1.43\times 10^{-63}(m-2)^{6+\frac{2}{10}+\frac{1}{100}+6\cdot10^{-6}}h^{0.999417} 
      \\ \geq &\frac{1.55\times 10^{11}}{(m-2)^{3+10^{-6}}}h^{1-10^{-6}}(\log(h))^{-5}-1.43\times 10^{-63}(m-2)^{6+\frac{2}{10}+\frac{1}{100}+6\cdot10^{-6}}h^{0.999417} .
    \end{align*}
    By bounding $$\log(h)\leq \frac{1}{r}h^r$$
    with $r=10^{-6}$ to obtain 
    $$S(A_{\textbf{1}},z) \geq \frac{1.55\times 10^{-19}}{(m-2)^{3+10^{-6}}}h^{1-6\cdot10^{-6}}-1.43\times 10^{-63}(m-2)^{6+\frac{2}{10}+\frac{1}{100}+6\cdot10^{-6}}h^{0.999417} .$$
    This is positive as long as 
    $$h > (9.22\times 10^{-45}(m-2)^{9.21})^{\frac{1}{5.77\times 10^{-4}}}.$$
    Since $a_j\geq 1,$ if the number of primes $p\not\in \{2,3\}$ dividing $x_{j_{0}}$ for some $1\leq j_{0} \leq 4$ is $\geq L,$ then 
    $$h=\sum_{j=1}^4a_j\left((m-2)x_j-(m-4)\right)^2\geq h^{\frac{2L}{1800}}$$
    and we conclude that $L\leq 900.$ This gives the first statement that 
    \[
    N_{L,S}\leq 1+(m-2)^{-1}(9.22\times 10^{-45}(m-2)^{9.21})^{\frac{1}{2\times 5.77\times 10^{-4}}}\leq C(m-2)^{7980}.
    \]

Now suppose that 
\begin{align*}
L&\geq \max\{900,1+7980\log_5(m-2)\},\\
h&\leq \left(9.22\times 10^{-45}(m-2)^{9.21}\right)^{\frac{1}{5.77\times 10^{-4}}},
\end{align*}
and every $n\leq C(m-2)$ is represented by $\sum_{j=1}^{\ell} a_j P_{m}x_j$ for some $C$ sufficiently large. Since every $n\leq C(m-2)$ is represented as a sum of $m$-gonal numbers, we know that the sum of $m$-gonal numbers is universal by work of Kim and Park \cite[Theorem 1.1]{KimPark}. If $x\in\Z^{\ell}$ is some representation of $h$ and is not a $P_{L,S}$-number, then 
\[
(m-2)5^L-(m-4)\leq (m-2)x_j-(m-4)\leq \frac{1}{a_j}\sqrt{h}\leq \left(9.22\times 10^{-45}(m-2)^{9.21}\right)^{\frac{10^4}{11.54}},
\]
Therefore 
\begin{align*}
L&\leq \log_5\left(1+(m-2)^{-1}\left(9.22\times 10^{-45}(m-2)^{9.21}\right)^{\frac{10^4}{11.54}}\right)\\
&\leq \log_5\left(\left(1+\left(9.22\times 10^{-45}\right)^{\frac{10^4}{11.54}} \right)(m-2)^{7980}\right)< 1+7980\log_5(m-2).
\end{align*}
This contradicts our assumption on $L$, so every $h\leq \left(9.22\times 10^{-45}(m-2)^{9.21}\right)^{\frac{1}{5.77\times 10^{-4}}}$ is represented as a sum with $x_j$ being $P_{L,S}$-numbers. Since every $h>\left(9.22\times 10^{-45}(m-2)^{9.21}\right)^{\frac{1}{5.77\times 10^{-4}}}$ is also represented with $P_{L,S}$ numbers from our calculation above (as long as $L\geq 900$), we have $P_{L,S}$-universality. 
\end{proof}

\end{document}